\documentclass[11pt]{article}

\usepackage{subfigure}
\usepackage{color}
\usepackage[cmex10]{amsmath}
\usepackage{mathrsfs,amssymb,amsmath}
\usepackage{cases}
\usepackage{graphicx}
\usepackage{caption}

\usepackage{epstopdf}
\usepackage{hyperref}

\usepackage{float}
\usepackage{placeins}

\newcommand{\R}{{\mathbb R}}

\def\dref#1{(\ref{#1})}

\usepackage{algorithm} 
\usepackage{algorithmic} 

\newtheorem{notation}{Notation}[section]
\newtheorem{corollary}{Corollary}[section]
\newtheorem{remark}{Remark}[section]
\newtheorem{lemma}{Lemma}[section]
\newtheorem{assumption}{Assumption}[section]
\newtheorem{proposition}{Proposition}

\newtheorem{theorem}{Theorem}[section]
\newtheorem{definition}{Definition}[section]

\newenvironment{proof}{{\bf {Proof.}}}{\hfill $\square$}
\allowdisplaybreaks[4]

\numberwithin{equation}{section}

\def\dref#1{(\ref{#1})}

\def\pt{\partial}

\def\ra{\rightarrow}

\def\s{\subseteq}

\def\e{\varepsilon}

\def\vp{\varphi}
\def\lg{\langle}

\def\rg{\rangle}

\def\mf{\mathfrak}
\def\bf{\textbf}
\def\pt{\partial}

\def\om{\omega}
\def\Om{\Omega}
\def\la{\lambda}
\def\al{\alpha}
\def\be{\beta}
\def\de{\delta}
\def\ga{\gamma}

\def\La{\Lambda}

\def\ts{\times}

\def\iy{\infty}

\def\f{\frac}

\def\Lra{\Leftrightarrow}

\def\ura{\rightharpoonup}

\def\df{\mathrm d}

\def\wt{\widetilde}
\def\wh{\widehat}

\def\esssup{\operatorname*{ess\ \! sup}}

\def\hra{\hookrightarrow}

\def\diag{{\rm diag}}

\def\mcA{\mathcal{A}}

	\def\mbT{\mathbb{T}}

	\DeclareMathOperator{\Div}{div}

	\newcommand{\N}{\mathbb N}

\begin{document}

\title{{\bf   {\bf   Null Controllability for Degenerate Parabolic Equations with Internal Control Applied on a Measurable Subset}}\footnote{\small This work was carried out with the support of the
National Natural Science Foundation of China under grant  nos. 12131008 and U23B2033, and
National Key R\&D Program of China under grant no. 2024YFA1013101.}}

\author{ Dong-Hui Yang$^{a}$, Mengze Gu$^{b}$ \footnote{\small
		The corresponding author: gu3059268737@gmail.com}, Bao-Zhu Guo$^{c}$,
		Ghadir Shokor$^{d}$
		\\
		$^a${\it School of Mathematics and Statistics, Central South University}\\
		{\it Changsha 410075, P.R.China}\\
		$^b${\it School of Mathematics and Science, Changsha Normal University}\\ {\it  Changsha 410083, China}\\
		$^c${\it Academy of Mathematics and Systems Science, Academia Sinica}\\
		{\it Beijing 100190, China}\\
		$^d${\it School of Mathematics and Statistics, Central South University}\\
			{\it Changsha 410075, P.R.China} }

\date{}

\maketitle{}

\begin{abstract}

  This work serves as a continuation of our preceding paper \cite{Yang2}. In that study, we presented a separable variable method to derive the Lebeau-Robbiano spectral inequality for a specific degenerate parabolic equation and subsequently employed it to demonstrate the null controllability of said equation when internal control is applied to an open subset. In the current paper, we reapply the separable variable method to attain the Lebeau-Robbiano spectral inequality for a different degenerate parabolic equation, and we substantiate the null controllability of this equation with internal control acting on a measurable subset. This approach may offer an alternative means of proving controllability results for degenerate parabolic equations.

\vspace{0.3cm}

	
\vspace{0.3cm}
		

	\end{abstract}

\section{Introduction}

	 In this paper, we explore the null controllability of the following degenerate equation:
\begin{equation}\label{12.14.1}
	\begin{cases}
		\partial_{t}\varphi - \Div(A \nabla \varphi) = \chi_D f, & \text{in } Q, \\
		\varphi = 0, & \text{on } \partial Q, \\
		\varphi(0) = \varphi^0, & \text{in } \Omega,
	\end{cases}
\end{equation}
where $w, \Omega, Q, D$, and $\varphi^0, f$ are defined in the subsequent Assumption \ref{Assumption (H)}.

\begin{assumption}[Assumptions and notations]\label{Assumption (H)}
Let $\mathbb{N} = \{0, 1, 2, \cdots\}$, $\mathbb{N}^* = \mathbb{N} \setminus \{0\}$. Define $\Omega = \mathbb{T} \times \Lambda$, where $\mathbb{T} = \{x \in \mathbb{R}^2 : |x| = 1\}$ is the torus and $\Lambda = (0, 1)$. For a given $T > 0$, set $Q = \Omega \times (0, T)$, $\Lambda_T = \Lambda \times (0, T)$, and $\partial Q = \partial \Omega \times (0, T)$. Let $D \subset Q$ be a measurable subset with positive measure. We denote the spatial variable by $z = (\theta, r) \in \Omega$ and the space-time variable by $(z, t) \in Q$.

Additionally, fix $\alpha \in (0, 1)$ as a given constant, set $w = r^\alpha$, and define $A = \diag(1, w) = \diag(1, r^\alpha)$ as the diagonal matrix. For the separation of variables discussed in the paper, we introduce the operators $\mathfrak{A}_1 u = -\partial_{\theta\theta} u$ and $\mathfrak{A}_2 u = -\partial_r(w \partial_r u) = -\partial_r(r^\alpha \partial_r u)$. The initial condition $\varphi^0 \in L^2(\Omega)$ and the control term $f \in L^2(Q)$. The solution of \eqref{12.14.1} is denoted by $\varphi(t) = \varphi(\cdot, t) = \varphi(\cdot, \cdot, t)$ when emphasizing the time dependence $t$, and by $\varphi(z, t; f)$ when emphasizing the dependence on the control $f$.
\end{assumption}

It is evident that the degenerate parabolic equation \eqref{12.14.1} is degenerate on the boundary $\mathbb{T} \times \{0\}$.

The controllability of partial differential equations (PDEs) is a well-established field with a rich history dating back to the 1950s. Numerous studies have been conducted on this topic, including works by \cite{Apraiz, Coron, FG, Fursikov, Lasiecka, Lions, Rousseau, Tenenbaum, Weiss, Zuazua}. More recently, research on the controllability of degenerate PDEs began in the early 2000s. To date, most studies have focused on one-dimensional equations \cite{Beauchard, Buffe, Cannarsa, Guo, Gueye, Liu}, with several addressing the two-dimensional case \cite{Araruna, Beauchard, Cannarsa1, Guo}. However, higher-dimensional degenerate equations remain largely unexplored \cite{Wu, Wu1, Yang1}.

To this end, controllability on measurable subsets (i.e., $L^\infty(Q)$ control) has been developed in \cite{Apraiz, Liu, Wang}. In our work, we rely on the Lebeau-Robbiano spectral inequality \cite{Apraiz, Buffe, Gueye, Liu, Rousseau, Wang, Yang2}, which applies to elliptic, parabolic, and hyperbolic equations. This inequality is derived from the Carleman estimate using pseudodifferential methods. Notably, Carleman estimates have a wide range of applications \cite{Araruna, Beauchard, Cannarsa1, FG, Fursikov, Guo, Rousseau, Wu, Yang2, Yang1, Yang3}, including unique continuation \cite{Cannarsa, Wu1, Yang1}, controllability, and inverse problems \cite{Bellassoued, Zuazua}.

In this context, the method of separation of variables is a fundamental technique for solving PDEs. For degenerate parabolic equations, it has found some applications \cite{Beauchard, Tenenbaum, Yang2}, although its use is typically restricted to special domains, such as rectangular domains \cite{Araruna} or the torus \cite{Tenenbaum}. The domain considered in this paper is also a special one. We note that, in general, the domain is neither rectangular nor a torus. In such cases, the method of separation of variables is no longer applicable \cite{Yang3}.

In this paper, we apply the method of separation of variables to derive the Lebeau-Robbiano inequality and then use analytic techniques to establish controllability on measurable subsets.

Our main results are stated in Theorems \ref{01.18.T1} and \ref{01.19.T1} below.
\begin{theorem}\label{01.18.T1}
Under Assumption \ref{Assumption (H)}, let $\alpha \in (0, 1)$ and $0 < a < b < 1$. Let $D \subset \Lambda \times (a, b) \times (0, T)$ be a measurable subset with positive measure. Then, there exists a constant $C \geq 1$, depending only on $\alpha, a, b, T$, and $D$, such that the solution of \eqref{12.14.1} with $f = 0$ satisfies
\begin{equation}\label{01.18.18}
	\begin{split}
		\left(\int_\Omega \varphi^2(T) \, dz\right)^{\frac{1}{2}} \leq C \iint_D |\varphi(z, t)| \, dz \, dt \quad \text{for all } \varphi^0 \in L^2(\Omega).
	\end{split}
\end{equation}
\end{theorem}

 Here and in the subsequent discussion, we employ the symbol
$C$ to represent constants, noting that these constants may assume different values across distinct contexts.

\begin{theorem}\label{01.19.T1}
Under the conditions in Theorem \ref{01.18.T1}, for any $\varphi^0 \in L^2(\Omega)$, there exists a control function $f \in L^\infty(Q)$ with
\begin{equation}\label{01.19.5}
	\|f\|_{L^\infty(\Omega \times (0, T))} \leq C \|\varphi^0\|_{L^2(\Omega)},
\end{equation}
such that the solution $\varphi$ of \eqref{12.14.1} satisfies
\begin{equation}\label{01.19.6}
	\varphi(z, T; f) = 0, \quad \text{a.e. } z \in \Omega.
\end{equation}
\end{theorem}

We proceed  as follows. In Section \ref{S2}, we introduce the solution spaces for equation \eqref{12.14.1} and discuss the
 Hilbert uniqueness method. In Section \ref{S3}, we apply the method of separation of variables to decompose equation \eqref{12.14.1} into a sequence of degenerate parabolic equations \eqref{01.16.1}. In Section \ref{S4}, we derive the Carleman estimate for the sequence \eqref{01.16.1} and obtain the corresponding spectral inequalities. Finally, in Section \ref{S5}, we prove our main results, Theorems \ref{01.18.T1} and \ref{01.19.T1}.
	
\section{Preliminary results}\label{S2}

 In this section, we introduce the weighted Sobolev spaces associated with equation \eqref{12.14.1}, define the concept of a weak solution for this equation, and discuss the J.-L. Lions Hilbert Uniqueness Method.

\subsection{Solution spaces}
	
Define
\begin{equation*}
	H^1(\Om;w)=\left\{u\in L^2(\Om)\colon \int_\Om \nabla u\cdot A\nabla u\df x<+\iy\right\}.
\end{equation*}
Its inner product and norm are given by
\begin{equation*}
	(u,v)_{H^1(\Om;w)}=\int_\Om (\nabla u\cdot A\nabla v+uv)\df x, \quad
	\|u\|_{H^1(\Om;w)}=(u,u)_{H^1(\Om;w)}^\f{1}{2}
\end{equation*}
for all $u,v\in H^1(\Om;w)$.
Set
\begin{equation*}
	H_0^1(\Om;w)=\mbox{the closure of the set $C_0^\iy(\Om)$ in }H^1(\Om;w).
\end{equation*}
and
\begin{equation*}
	H^{-1}(\Om;w)=\mbox{the dual space of }H_0^1(\Om;w).
\end{equation*}

Define
\begin{equation*}
	H^2(\Om;w)=\left\{u\in H^1(\Om;w)\colon \mcA u\in L^2(\Om)\right\}.
\end{equation*}
Its inner product and norm are defined by
\begin{equation*}
	(u,v)_{H^2(\Om;w)}=\int_\Om (\mcA u)(\mcA v)\df x+(u,v)_{H^1(\Om;w)},\quad \|u\|_{H^2(\Om;w)}=(u,u)_{H^2(\Om;w)}^\f{1}{2}
\end{equation*}
for all $u,v\in H^2(\Om;w)$. Set
\begin{equation*}
	D(\mcA)=H_0^1(\Om;w)\cap H^2(\Om;w).
\end{equation*}

It is well-known that
\begin{equation*}
	(H_0^1(\Om;w),(\cdot,\cdot)_{H_0^1(\Om;w)}), \mbox{ and } (H^1(\Om;w), (\cdot,\cdot)_{H^1(\Om;w)}), \mbox{ and } (H^2(\Om;w),(\cdot,\cdot)_{H^2(\Om;w)})
\end{equation*}
are Hilbert spaces (see, e.g., \cite{GC,Heinonen}.)

 The following lemma presents a Hardy inequality.

\begin{lemma}\label{01.04.L1}
	Let $\al\in (0,1)$. Then there exists a constant $C>0$, depending only on $\al$, such that
	\begin{equation*}
		\int_\Om r^{\al-2}u^2\df z\leq C\int_\Om r^\al (\pt_ru)^2\df z
	\end{equation*}
	for all $u\in H_0^1(\Om;w)$.
\end{lemma}

\begin{proof}
	For each $\be\in (\al,1)$, we compute
	\begin{equation*}
		\begin{split}
			\int_\Om r^{\al-2}u^2\df z
			&=\int_\mbT\int_0^1r^{\al-2}|u(\theta,r)-u(\theta,0)|^2\df r\df \theta=\int_\mbT \int_0^1r^{\al-2}\left|\int_0^r\pt_r u(\theta,s)\df s\right|^2\df r\df \theta\\
			&\leq \int_\mbT\int_0^1r^{\al-2}\left(\int_0^rs^\be |\pt_ru(\theta,s)|^2\df s\right)\left( \int_0^rs^{-\be}\df s\right)\df r\df \theta\\
			&\leq \f{1}{1-\be}\int_\mbT \int_0^1\int_0^r r^{\al-\be-1}s^\be |\pt_r u(\theta,s)|^2\df s\df r\df \theta\\
			&=\f{1}{1-\be}\int_\mbT \int_0^1\int_s^1r^{\al-\be-1}s^\be |\pt_ru(\theta,s)|^2\df r\df s\df \theta\\
			&\leq \f{1}{(1-\be)(\be-\al)}\int_\mbT \int_0^1 s^\al |\pt_ru(\theta,s)|^2\df s\df \theta\\
			&=\f{1}{(1-\be)(\be-\al)}\int_\Om r^\al |\pt_ru(\theta,s)|^2\df z.
		\end{split}
	\end{equation*}
	 Taking $\be=\f{1+\al}{2}$, we   complete the proof of the lemma.
\end{proof}

\begin{remark}\label{01.06.R1}
	From Lemma \ref{01.04.L1} we have
	\begin{equation*}
		\int_\Om u^2\df z\leq  \int_\Om r^{\al-2}u^2\df z\leq C\int_\Om \nabla u\cdot A\nabla u\df z,
	\end{equation*}
	where the constant   $C>0$ depends only on $\al$.  Hence, in what follows, we use the inner product and norm
	\begin{equation*}
		(u,v)_{H_0^1(\Om;w)}=\int_\Om \nabla u\cdot A\nabla v\df x,\quad \|u\|_{H_0^1(\Om;w)}=(u,u)_{H_0^1(\Om;w)}^\f{1}{2}
	\end{equation*}
	on $H_0^1(\Om;w)$, respectively.
\end{remark}

\begin{lemma}\label{01.05.L4}
	Let $\al\in (0,1)$. Then,  the embedding $H_0^1(\Om;w)\hra L^2(\Om)$ is compact.
\end{lemma}

\begin{proof}
	 Assume that the sequence $\{u_n\}_{n\in\N}\s H_0^1(\Om;w)$ is a bounded sequence.
 Our objective is to demonstrate the existence of a subsequence of  $\{u_n\}_{n\in\N}$ that converges strongly in $L^2(\Om)$.
  Given that  $H_0^1(\Om;w)$ is a Hilbert space, there exists a subsequence of $\{u_n\}_{n\in\N}$, which is
  still denoted by itself, and an element $u_0\in H_0^1(\Om;w)$ such that $u_n\ura u_0$ weakly in $H_0^1(\Om;w)$.
   To prove that  $u_n\ra u$ strongly in $L^2(\Om)$, it suffices to consider an abstract subsequence and establish the result for the case
     $u_0=0$  (if this is not the case, we can replace  $u_n$ with $u_n-u_0$ for all $n\in\N$).
	
	Let $\e>0$ be arbitrary. On the one hand, for each $\de\in (0,\f{1}{4})$, define
 $\Om_\de=\mbT\ts (\de,1)$. By  Lemma \ref{01.04.L1}, we have
	\begin{equation*}
		\begin{split}
			\int_{\Om-\Om_\de}u^2\df x
			&=\int_{\Om-\Om_\de}r^{2-\al}r^{\al-2}u^2\df x\leq \de^{2-\al}\int_\Om r^{\al-2}u^2\df x\\
			&\leq C\de^{2-\al}\int_\Om \nabla u\cdot A\nabla u\df x=C\de^{2-\al}\|u\|_{H_0^1(\Om;w)}^2
		\end{split}
	\end{equation*}
	for all $u\in H_0^1(\Om;w)$. Then, we select $\de_0>0$   sufficiently small such that for all
 $\de\in (0,\de_0]$ such that $\int_{\Om-\Om_\de}u_n^2\df x<\f{1}{4}\e^2$ for all $n\in\N$. On the other hand, for this fixed $\de_0$,
   since the embedding $H^1(\Om_{\de_0};w|_{\Om_{\de_0}})=H^1(\Om_{\de_0})\hra L^2(\Om_{\de_0})$ is compact and  $u_n\ura 0$ weakly in $H^1(\Om_{\de_0};w|_{\Om_{\de_0}})$,   there exists at least one $n_\e\in\N$ such that $\|u_{n_\e}\|_{L^2(\Om_{\de_0})}<\f{1}{2}\e$.  Combining these two facts, we conclude that $\|u_{n_\e}\|_{L^2(\Om)}<\e$. This completes the proof of the lemma.
\end{proof}

\begin{notation}\label{01.05.N1}
According to Lemma \ref{01.05.L4}, the partial differential operator
$\mcA$
 possesses a discrete point spectrum given by	
	\begin{equation*}
		0<\la_1< \la_2\leq \la_3\leq \cdots \ra +\iy,
	\end{equation*}
	 which means that $\la_n\in\R$ satisfies the following equation
	\begin{equation}\label{01.06.7}
		\begin{cases}
			\mcA\Phi_n=\la_n\Phi_n, &\mbox{in }\Om,\\
			\Phi_n=0, &\mbox{on }\pt\Om.
		\end{cases}
	\end{equation}
	We denote $\{\Phi_n\}_{n\in\N}\s D(\mcA)$ the normalized basis for $L^2(\Om)$, i.e., $\|\Phi_n\|_{L^2(\Om)}=1$ for all $n\in\N$
(see e.g., \cite[Theorem 7, Appendix D.6, p.728]{Evans}.
\end{notation}

From Remark \ref{01.06.R1}, we obtaion
\begin{equation*}
	\f{(1-\al)^2}{4}<\la_1=\inf_{0 \neq u\in H_0^1(\Om;w)}\f{\int_\Om  \nabla u\cdot A\nabla u\df z}{\int_\Om u^2\df z}.
\end{equation*}

\begin{definition}\label{01.06.D1}
	A function
	\begin{equation*}
		\vp\in L^2(0,T; H_0^1(\Om;w)), \mbox{ with } \pt_t\vp\in L^2(0,T; H^{-1}(\Om;w))
	\end{equation*}
	 is considered a weak solution of the equation \eqref{12.14.1} with respect to $(\vp^0,f)$, provided that
	
	(i) for each $v\in H_0^1(\Om;w)$  and almost every  $t\in [0,T]$, we have
	\begin{equation*}
		\lg \pt_t\vp, v\rg_{H^{-1}(\Om;w), H_0^1(\Om;w)}+(\vp, v)_{H_0^1(\Om;w)}=(f,v)_{L^2(\Om)},
	\end{equation*}
	
	(ii) $\vp(0)=\vp^0$.
\end{definition}

\begin{theorem}\label{01.06.T1}
	 Under Assumption \ref{Assumption (H)}, the equation \eqref{12.14.1} with respect to  $(\vp^0,f)$
admits a unique weak solution
	\begin{equation*}
		\vp\in L^2(0,T; H_0^1(\Om;w)), \mbox{ with } \pt_t\vp\in L^2(0,T; H^{-1}(\Om;w))
	\end{equation*}
	 that satisfies the following estimate:
	\begin{equation*}
		\esssup_{t\in [0,T]}\|\vp(t)\|_{L^2(\Om)}+\|\vp\|_{L^2(0,T;H_0^1(\Om;w))}+\|\pt_t\vp\|_{L^2(0,T; H^{-1}(\Om;w))}\leq C\left(\|f\|_{L^2(Q)}+\|\vp^0\|_{L^2(\Om)}\right),
	\end{equation*}
	where the constant $C>0$ depends only on $\al$ and $T$.
	
	 Furthermore, if we additionally assume $\vp^0\in H_0^1(\Om;w)$, then the weak solution
	\begin{equation*}
		\vp\in L^2(0,T; D(\mcA))\cap C([0,T]; H_0^1(\Om;w)), \mbox{ with } \pt_t\vp\in L^2(Q),
	\end{equation*}
	and has the estimate
	\begin{equation*}
		\esssup_{t\in [0,T]}\|\vp(t)\|_{H_0^1(\Om;w)}+\|\vp\|_{L^2(0,T; D(\mcA))}+\|\pt_t\vp\|_{L^2(Q)}\leq C\left(\|f\|_{L^2(Q)}+\|\vp^0\|_{H_0^1(\Om;w)}\right),
	\end{equation*}
	where the constant $C>0$ depends only on $\al$ and $T$.
\end{theorem}

\begin{proof}
	 The proof of this theorem is accomplished using the classical Galerkin method.
\end{proof}

\subsection{Null controllability, observability and HUM}

 Here, we introduce the Hilbert Uniqueness Method (HUM).

\begin{definition}\label{01.17.D1}
	 We say that the system \eqref{12.14.1} is null controllable if there exists a control   $f\in L^2(Q)$ such that $\vp(T)=0$
 throughout
  $\Om$.
\end{definition}

\begin{definition}
	 We refer to the system
	\begin{equation}\label{01.17.11}
		\begin{cases}
			\pt_t\vp+\Div(A\nabla \vp)=0, &\mbox{in }Q, \\
			\vp=0, &\mbox{on }\pt Q, \\
			\vp(0)=\vp^0, &\mbox{in }\Om
		\end{cases}
	\end{equation}
	 as being observable on   $\om\Subset\Om$ provided that there exists a constant
 $C>0$, depending only on $\al, T, \om$ and $\Om$,   such that the following estimate holds:
	\begin{equation*}
		\int_\Om \vp^2(T)\df z\leq C\iint_{\om\ts (0,T)} \vp^2\df z\df t.
	\end{equation*}
\end{definition}

 The following theorem presents the HUM.

\begin{theorem}\label{01.17.T4}
	 The null controllability of the system \eqref{12.14.1} is equivalent to the observability of the system \eqref{01.17.11}.
\end{theorem}

\begin{proof}
	 The desired result follows from \cite[Theorem 2.42, p. 56]{Coron}.
\end{proof}

\section{Separate variables}\label{S3}

 In this section, we employ the method of separable variables to transform equation \eqref{12.14.1} into a series of degenerate parabolic equations \eqref{01.16.1}.

We recall that $\{g_{1,n}=\f{1}{\sqrt {2\pi}}\cos n\theta\}_{n\in\N}\cup \{g_{2,n}=\f{1}{\sqrt {2\pi}}\sin n\theta\}_{n\in\N^*}$ are the orthonormal basis of $L^2(\mbT)$, and
\begin{equation*}
	H^k(\mbT)=\left\{\sum_{n\in\N} a_{1,n} g_{1,n}+\sum_{n\in\N^*} a_{2,n}g_{2,n}\colon (a_{1,0})^2+\sum_{n\in\N^*} n^{2k}\left((a_{1,n})^2+(a_{2,n})^2\right) <+\iy\right\}
\end{equation*}
with the norm defined as
\begin{equation*}
	\|u\|_{H^k(\mbT)}^2=(a_{1,0})^2+\sum_{n\in\N^*} n^{2k}\left((a_{1,n})^2+(a_{2,n})^2\right), \mbox{ with } u=\sum_{n\in\N} a_{1,n} g_{1,n}+\sum_{n\in\N} a_{2,n}g_{2,n}.
\end{equation*}

 We formally assume
\begin{equation}\label{01.16.3}
	\vp(\theta, r,t)=\sum_{n\in\N} \vp_{1,n}(r,t)g_{1,n}(\theta)+\sum_{n\in\N^*}\vp_{2,n}(r,t)g_{2,n}(\theta),
\end{equation}
 which leads to
\begin{equation}\label{01.16.1}
	\begin{split}
		\pt_t\vp_{i,n}+\mf{A}_2\vp_{i,n}+n^2\vp_{i,n}=f_{i,n}, \ n\in\N,\ i=1,2,
	\end{split}
\end{equation}
where
\begin{equation*}
	f_{i,n}=(f,g_{i,n})_{L^2(\mbT)}, \ n\in \N,\ i=1,2.
\end{equation*}

 Next, we will analyze the equations \eqref{01.16.1} for all $n\in\N, i=1,2$.  To this end, we introduce the weighted Sobolev spaces.

Define
\begin{equation*}
	H^1(\La;w)=\left\{u\in L^2(\La)\colon \int_\La w(\pt_ru)^2\df r<+\iy\right\}
\end{equation*}
 with the inner product and norm given by
\begin{equation*}
	(u,v)_{H^1(\La;w)}=\int_\La \big(uv+w(\pt_ru)(\pt_r v)\big)\df r, \quad \|u\|_{H^1(\La;w)}=(u,u)_{H^1(\La;w)}^\f{1}{2},
\end{equation*}
respectively. Set
\begin{equation*}
	\begin{split}
	H_0^1(\La;w)
	&=\mbox{the closure of $C_0^\iy(\La)$ in }H^1(\La;w), \\
	H^{-1}(\La;w)
	&=\mbox{the dual space of } H_0^1(\La;w).
	\end{split}
\end{equation*}

Define
\begin{equation*}
	H^2(\La;w)=\left\{u\in H^1(\La;w)\colon \mf{A}_2u\in L^2(\La)\right\}
\end{equation*}
 with the inner product and norm defined as
\begin{equation*}
	(u,v)_{H^2(\La;w)}=\int_\La (\mf{A}_2u)(\mf{A}_2v)\df r+(u,v)_{H^1(\La;w)},\quad \|u\|_{H^2(\La;w)}=(u,u)_{H^2(\La;w)}^\f{1}{2},
\end{equation*}
respectively. Set
\begin{equation*}
	D(\mf{A}_2)=H_0^1(\La;w)\cap H^2(\La;w).
\end{equation*}

\begin{lemma}\label{01.16.L2}
	Let $\al\in (0,1)$. Then there exists a constant $C>0$, depending only on $\al$, such that
	\begin{equation*}
		\int_\La r^{\al-2}u^2\df r\leq C\int_\La r^\al (\pt_ru)^2\df r
	\end{equation*}
	for all $u\in H_0^1(\La;w)$.
\end{lemma}

\begin{proof}
	 The proof follows similarly to that of Lemma \ref{01.04.L1}.
\end{proof}

\begin{remark}\label{01.16.R11}
	From Lemma \ref{01.16.L2}, we have
	\begin{equation*}
		\int_\La u^2\df r\leq \int_\La r^{\al-2}u^2\df r\leq C\int_\La r^\al (\pt_ru)^2\df r,
	\end{equation*}
	where the constant $C>0$ depends only on $\al$.  Hence, in the following, we utilize the inner product and norm
	\begin{equation*}
		(u,v)_{H_0^1(\La;w)}=\int_\La w(\pt_ru)(\pt_rv)\df r,\quad \|u\|_{H_0^1(\La;w)}=(u,u)_{H_0^1(\La;w)}^\f{1}{2}
	\end{equation*}
	on $H_0^1(\La;w)$, respectively.
\end{remark}

\begin{lemma}\label{01.16.L3}
	Let $\al\in (0,1)$. Then the embedding $H_0^1(\La;w)\hra L^2(\La)$ is compact.
\end{lemma}

\begin{proof}
	 The proof is analogous to that of Lemma \ref{01.05.L4}.
\end{proof}

\begin{notation}\label{01.06.N2} We define
\begin{equation}\label{05.20.1}
	H^1(\mathbb{T})\otimes H_0^1(\Lambda;w)=\left\{\varphi_1(\theta)\varphi_2(r)\colon \varphi_1(\theta)\in H^1(\mathbb{T}), \varphi_2(r)\in H_0^1(\Lambda;w)\right\},
\end{equation}
and
\begin{equation*}
	\begin{split}
		H^1(\mathbb{T})\widehat{\otimes} H_0^1(\Lambda;w)=\overline{\text{Span}\left[H^1(\mathbb{T})\otimes H_0^1(\Lambda;w)\right]}^{\|\cdot\|_{H^1(\Omega;w)}},
	\end{split}
\end{equation*}
where $\text{Span}\ \! H$ denotes the linear span of the space $H$.
\end{notation}

  The following Lemma \ref{05.19.L2} addresses a fundamental issue in Sobolev spaces. It is known that
  $H^1(\mbT)\wh\otimes H_0^1(\Lambda)=H_0^1(\Omega)$, as demonstrated in \cite[Section 9 in Chapter V]{Hilbert}. Specifically,
   the set $\{g_{i,n}(\f{1}{\sqrt{2\pi}}\sin m\pi \theta)\}\subset H^1(\mbT)\otimes H_0^1(\Lambda)$ forms an orthogonal basis
   for  $H_0^1(\Omega)$.

\begin{lemma}\label{05.19.L2}
		$H_0^1(\Omega;w)=H^1(\mbT )\widehat{\otimes} H_0^1(\Lambda;w)$.
\end{lemma}

\begin{proof}
	For every $\vp_1\in H^1(\mbT)$ and $\vp_2\in H_0^1(\La;w)$, we observe:
	\begin{equation*}
		\begin{split}
			\int_\Omega \nabla(\varphi_1\varphi_2)\cdot  A\nabla(\varphi_1\varphi_2)\df x
			&=\int_\Omega (\pt_\theta\vp_1)^2(\varphi_2)^2dz+\int_\Omega (\varphi_1)^2\left[ r^{\alpha}(\partial_{r}\varphi_2)^2\right]dz\\
			&\leq 2\|\varphi_1\|_{H^1(\mbT)}^2\|\varphi_2\|_{H_0^1(\La;w)}^2,
		\end{split}
	\end{equation*}
	which implies $\varphi_1\varphi_2\in H_0^1(\Omega;w)$. Consequently, $H^1(\mbT)\otimes H_0^1(\Lambda;w)\subset H_0^1(\Omega;w)$. Hence, $H^1(\mbT)\widehat{\otimes} H_0^1(\Lambda;w)\subset H_0^1(\Omega;w)$.

We now prove the converse implication.
Since $C_0^\infty(\Omega)\subset H_0^1(\Omega)\subset H_0^1(\Omega;w)$, and $H_0^1(\Omega)= H^1(\mathbb{T})\widehat{\otimes} H_0^1(\Lambda)$ from Fourier analysis, and $H^1(\mathbb{T})\otimes H_0^1(\Lambda)\subset H^1(\mathbb{T})\widehat{\otimes} H_0^1(\Lambda;w)$, we obtain
\begin{equation*}
	H_0^1(\Omega;w)\subset H^1(\mathbb{T})\widehat{\otimes} H_0^1(\Lambda;w).
\end{equation*}
This completes the proof of the lemma.	
\end{proof}

\begin{remark}\label{01.16.R2}
	 By Lemma \ref{05.19.L2}, the formal decomposition \eqref{01.16.3} is well-defined.
\end{remark}

 Now, we consider the following equations \eqref{01.16.1} for all $n\in\N$ and $i=1,2$.

\begin{notation}\label{01.16.N1}
	 From Lemma \ref{01.16.L3}, the partial differential operator $\mf{A}_2$ has discrete point spectrum
	\begin{equation*}
		0<\la_{2,1}< \la_{2,2}< \la_{2,3}< \cdots \ra +\iy,
	\end{equation*}
	i.e., $\la_{2,n}\in\R$ satisfies the following equation
	\begin{equation}\label{01.16.2}
		\begin{cases}
			\mf{A}_2\Phi_{2,n}=\la_{2,n}\Phi_{2,n}, &\mbox{in }\Om,\\
			\Phi_{2,n}=0, &\mbox{on }\pt\Om.
		\end{cases}
	\end{equation}
	We denote $\{\Phi_{2,n}\}_{n\in\N}\s D(\mf{A}_2)$ as the normalized basis of
 $L^2(\Om)$, i.e., $\|\Phi_{2,n}\|_{L^2(\Om)}=1$ for all $n\in\N$ (see, e.g., \cite[Theorem 7, Appendix D.6, p. 728]{Evans}).
\end{notation}

From Remark \ref{01.16.R11}, we have
\begin{equation*}
	\f{(1-\al)^2}{4}<\la_{2,1}=\inf_{0 \neq u\in H^1_0(\La;w)}\f{\int_\La  w(\pt_ru)^2\df r}{\int_\La u^2\df r}.
\end{equation*}

\begin{remark}\label{01.16.R3}
	Let $\la\in\R$ be the eigenvalue of $\mf{A}_2+n^2$, and $\Psi\in H_0^1(\La;w)$ be the corresponding eigenfunction  of $\mf{A}_2+n^2$, i.e.,
	\begin{equation*}
		\begin{cases}
			(\mf{A}_2+n^2)\Psi=\la\Psi, &\mbox{in }\La, \\
			\Psi=0, &\mbox{on } \La.
		\end{cases}
	\end{equation*}
	Then $\mf{A}_2\Psi=(\la-n^2)\Psi$. This  implies that $\la-n^2$ is an eigenvalue of $\mf{A}_2$, and $\Psi$ is the
 corresponding eigenfunction of $\mf{A}_2$.  Conversely, if $\la$ is the eigenvalue of $\mf{A}_2$, and $\Psi$ is the eigenfunction of $\mf{A}_2$ with respect to $\la$, then $\la+n^2$ is the eigenvalue of $\mf{A}_2+n^2$, and $\Psi$ is the corresponding
 eigenfunction of $\mf{A}_2+n^2$.
\end{remark}

\section{Observability inequalities on special control domains and spectral inequalities}\label{S4}

 In this section, we derive the observability inequality  and the spectral inequality. These findings will serve as crucial tools in Section \ref{S5}.

Theorem \ref{01.17.T1} and Corollary \ref{01.17.C1} are analogous to the results in \cite{Yang2}, and we here provide sharper estimates.

\subsection{Carleman estimate}

Consider the following equation for
 $n\in\N$ and $i=1,2$:
\begin{equation}\label{01.17.1}
	\begin{cases}
		\pt_t\vp_{i,n}+\mf{A}_2\vp_{i,n}+n^2\vp_{i,n}=f_{i,n}, &\mbox{in }\La_T,\\
		\vp_{i,n}=0, &\mbox{on } \pt\La_T, \\
		\vp_{i,n}(0)=\vp_{i,n}^0, &\mbox{in }\La,
	\end{cases}
\end{equation}
where $\vp_{i,n}^0\in L^2(\La)$ and $f_{i,n}\in L^2(Q)$.

\begin{notation}\label{01.17.N1}
Let  $\eta \in C^3[a,1]$ be defined as follows:
\begin{equation*}
	\eta=
	\begin{cases}
		r^{2-\al}, &\mbox{on }(0, \f{2a+b}{3}), \\
		>0, &\mbox{on }(\f{2a+b}{3}, \f{a+2b}{3}),\\
		(1-r)r^{-\al}, &\mbox{on }(\f{a+2b}{3},1),
	\end{cases}
\end{equation*}
and
\begin{equation*}
	\Theta=\f{1}{[t(T-t)]^4},\quad \xi= \Theta \left(\ga-\eta\right),
\end{equation*}
where $\gamma \geq |\eta|_\iy+1\ (|\eta|_\iy=|\eta|_{L^\iy(\La)})$  is a constant to be specified later.
Define
\begin{equation*}
	\psi=e^{-s\xi}\vp_{i,n}\ \Lra \ \vp_{i,n}=e^{s\xi}\psi
\end{equation*}
for $s\geq 1$.
\end{notation}

From Notation \ref{01.17.N1}, we obtain
\begin{equation}\label{01.17.2}
	\psi(t)=\nabla\psi(t)=0 \mbox{ for } t=0, T,
\end{equation}
and
\begin{equation}\label{01.17.3}	
	\psi=\pt_t\psi=0 \mbox{ on }\pt \La_T.
\end{equation}\label{01.17.4}
 Furthermore,
\begin{equation}\label{01.17.5}
	\begin{split}
		\xi_r=-\Theta\eta_r,\quad \xi_t=\Theta'(\ga-\eta),\quad \xi_{rt}=-\Theta'\eta_r, \quad \xi_{tt}=\Theta''(\ga-\eta),
	\end{split}
\end{equation}
and
\begin{equation}\label{01.17.6}
	\begin{split}
		(w\xi_r)_r=-\ga\Theta (w\eta_r)_r, \quad
	\end{split}
\end{equation}
 Equation \eqref{01.17.1} can be rewritten as
\begin{equation}
	e^{-s\xi}f_{i,n}=P_1\psi+P_2\psi,
\end{equation}
where
\begin{equation}\label{gbz1}
	\left\{\begin{array}{ll}
		P_1\psi
		&=\sum_{i=1}^3P_{1i}\psi=\psi_t-2sw\psi_r\xi_r-s\psi (w\xi_r)_r, \\
		P_2\psi
		&=\sum_{i=1}^4P_{2i}\psi=-(w\psi_r)_r+s\psi\xi_t-s^2\psi w(\xi_r )^2+n^2\psi.
	\end{array}\right. 
\end{equation}

\subsubsection{Computation}

  All subsequent applications of integration by parts are valid; for further particulars, refer to \cite[Proposition 2.4, p. 174]{Alabau}. It is important to note that the term $n^2 \psi$ has no impact on the calculations or the subsequent estimations.
  
Firstly, we  calculate $(P_{11}\psi, P_2\psi)_{L^2(\La_T)}$ defined by \dref{gbz1}. 
Indeed, based on \eqref{01.17.2} and \eqref{01.17.3}, we obtain 
\begin{equation*}
\begin{split}
(P_{11}\psi, P_2\psi){L^2(\La_T)}
&=-\frac{s}{2}\iint_{\La_T}\Theta''\psi^2(\gamma-\eta)\mathrm{d} r\mathrm{d} t+s^2\iint_{\La_T}\Theta \Theta' \psi^2w\eta_r^2\mathrm{d} r\mathrm{d} t.
\end{split}
\end{equation*}

Secondly, we calculate $(P_{12}\psi+P_{13}\psi, P_{21}\psi)_{L^2(\La_T)}$.
Indeed, from \eqref{01.17.3} and \eqref{01.17.5}, we derive
 \begin{equation*}
	\begin{split}
		(P_{12}\psi+P_{13}\psi, P_{21}\psi)_{L^2(\La_T)}
		&=-s\int_0^T\Theta \eta_r (w\psi_r)^2 \bigg|_{r=0}^{r=1} \df t +s\iint_{\La_T}\Theta w\psi\psi_r(w\eta_r)_{rr}\df r\df t\\
		&\hspace{4.5mm}+2s\iint_{\La_T}\Theta w\psi_r^2(w\eta_r)_r\df r\df t-s\iint_{\La_T} \Theta w\psi_r^2\eta_rw_r\df r\df t.
	\end{split}
\end{equation*}

Thirdly, we calculate $(P_{12}\psi+P_{13}\psi, P_{22}\psi)_{L^2(\La_T)}$.
Indeed, from \eqref{01.17.3}, we find 
\begin{equation*}
\begin{split}
(P_{12}\psi+P_{13}\psi, P_{22}\psi)_{L^2(\La_T)}=s^2\iint_{\La_T}\Theta \Theta' \psi^2w\eta_r^2\mathrm{d} r\mathrm{d} t.
\end{split}
\end{equation*}

Fourthly, we calculate $(P_{12}\psi+P_{13}\psi, P_{23}\psi)_{L^2(\La_T)}$.
Indeed, from \eqref{01.17.3} and \eqref{01.17.6}, we obtain
\begin{equation*}
\begin{split}
(P_{12}\psi+P_{13}\psi, P_{23}\psi)_{L^2(\La_T)}=s^3\iint_{\La_T} \Theta^3\psi^2 (w\eta_r)\left(w\eta_r^2\right)_r\mathrm{d} r\mathrm{d} t.
\end{split}
\end{equation*}
Finally, direct calculation using \eqref{01.17.3} yields $(P_{12}\psi+P_{13}\psi, P_{24}\psi)_{L^2(\La_T)}=0$.
\vspace{3mm}

\subsubsection{Estimates}

  We define the following regions:
\begin{equation*}
\La_T^1=\left(0, \frac{2a+b}{3}\right)\times (0,T), \quad \La_T^3=\left(\frac{a+2b}{3}, 1\right)\times (0,T),\quad \La_T^2=\La_T-(\La_T^1\cup \La_T^3).
\end{equation*}
We will estimate $(P_1\psi,P_2\psi){L^2(\La_T)}$ defined by \dref{gbz1} using the following decomposition:
\begin{equation*}
E_i=(P_1\psi,P_2\psi)_{L^2(\La_T^i)}, \quad i=1,2,3.
\end{equation*}
Note that the following inequalities hold:
\begin{equation}\label{01.17.7}
|\Theta'|\leq 12 T\Theta^{\frac{5}{4}}, \quad |\Theta''|\leq 196T^2\Theta^{\frac{3}{2}}.
\end{equation}
Additionally, from the condition $r^{1+\alpha}(\psi_r)^2=r^{1-\alpha}(r^\alpha \psi_r)^2=0$ at $r=0$, we have
\begin{equation*}
\begin{split}
-s\int_0^T\Theta \eta_r(w\psi_r)^2 \bigg|_{r=0}^{r=1}  dt\geq 0.
\end{split}
\end{equation*}

Firstly, we  compute $E_1$. On $\La_T^1$, we have 
\begin{equation*}
\begin{split}
\eta_r=(2-\alpha)r^{1-\alpha}, \quad w\eta_r^2=(2-\alpha)^2r^{2-\alpha},\quad (w\eta_r)r=2-\alpha, \quad (w\eta_r){rr}=0,
\end{split}
\end{equation*}
and
\begin{equation*}
\eta_rw_r=(2-\alpha)\alpha,\quad (w\eta_r)(w\eta_r^2)_r=(2-\alpha)^4r^{2-\alpha}.
\end{equation*}
From \eqref{01.17.7} and the inequalities
\begin{equation*}
	\begin{split}
		s\ga\iint_{\La_T^1}\Theta^\f{3}{2}\psi^2\df r\df t
		&\leq \f{1}{2}s^2\ga^2\iint_{\La_T^1}\Theta^2r^{2-\al}\psi^2\df r\df t+\f{1}{2}\iint_{\La_T^1}\Theta r^{\al-2}\psi^2\df r\df t\\
		&\leq \f{3}{2}s^2\ga^2T^2\iint_{\La_T^1}\Theta^3r^{2-\al}\psi^2\df r\df t+\f{2}{(1-\al)^2}\iint_{\La_T}\Theta w(\psi_r)^2\df r\df t
	\end{split}
\end{equation*}
and
\begin{equation*}
	\begin{split}
		\iint_{\La_T^1}\Theta^\f{9}{4}r^{2-\al}\psi^2\df r\df t\leq 8T^6\iint_{\La_T^1}\Theta^3r^{2-\al}\psi^2\df r\df t,
	\end{split}
\end{equation*}
we obtain 
\begin{equation*}
	\begin{split}
		E_1
		&\geq (2-\al)^2s\iint_{\La_T^1}\Theta w\psi_r^2\df r\df t+(2-\al)^4s^3\iint_{\La_T^1}\Theta^3r^{2-\al}\psi^2\df r\df t\\
		&\hspace{4.5mm}-Cs\ga T^2\iint_{\La_T^1}\Theta^\f{3}{2}\psi^2\df r\df t-Cs^2T\iint_{\La_T^1}\Theta^\f{9}{4}r^{2-\al}\psi^2\df r\df t\\
		&\geq (2-\al)^2s\iint_{\La_T^1}\Theta w\psi_r^2\df r\df t+(2-\al)^4s^3\iint_{\La_T^1}\Theta^3r^{2-\al}\psi^2\df r\df t\\
		&\hspace{4.5mm}-Cs^2\ga^2 T^4\iint_{\La_T^1}\Theta^3r^{2-\al}\psi^2\df r\df t-Cs^2T^7\iint_{\La_T^1}\Theta^3r^{2-\al}\psi^2\df r\df t\\
		&\hspace{4.5mm}-CT^2\iint_{\La_T}\Theta w(\psi_r)^2\df r\df t\\
		&\geq (2-\al)^2s\iint_{\La_T^1}\Theta w\psi_r^2\df r\df t+(2-\al)^4s^3\iint_{\La_T^1}\Theta^3r^{2-\al}\psi^2\df r\df t\\
		&\hspace{4.5mm}-CT^2\iint_{\La_T^2\cup \La_T^3}\Theta w(\psi_r)^2\df r\df t
	\end{split}
\end{equation*}
for $s\geq C\gamma^2 \max\{1,T^7\}$,
where the constant $C>0$ depends only on $\alpha$.

Secondly, we compute $E_2$.
From \eqref{01.17.7} and the inequality
\begin{equation*}
s\iint_{\La_T^2}\Theta \psi\psi_r dr dt\leq C\iint_{\La_T^2}\Theta \psi_r^2 dr dt+Cs^2T^{16}\iint_{\La_T^2}\Theta^3\psi^2 dr dt
\end{equation*}
we get
\begin{equation*}
\begin{split}
E_2
&\geq -Cs\iint_{\La_T^2}\Theta w\psi_r^2 dr dt-Cs^3\iint_{\La_T^2}\Theta^3\psi^2 dr dt
\end{split}
\end{equation*}
for $s\geq C\sqrt{\gamma}\max\{1,T^{16}\}$, where the constant $C>0$ depends only on $\alpha, a, b$, and $\eta|_{(a,b)}$.

Thirdly, we  compute $E_3$. On $\La_T^3$, we have 
\begin{equation*}
\begin{split}
\eta_r=-r^{-\alpha}\left[(1-\alpha)+\alpha r^{-1}\right], \quad w\eta_r^2=r^{-\alpha}\left[(1-\alpha)+\alpha r^{-1}\right]^2, \quad (w\eta_r)r=\alpha r^{-2},
\end{split}
\end{equation*}
\begin{equation*}
\begin{split}
(w\eta_r){rr}=-2\alpha r^{-3}, \quad \eta_rw_r=-\alpha r^{-1}\left[(1-\alpha)+\alpha r^{-1}\right],
\end{split}
\end{equation*}
and
\begin{equation*}
(w\eta_r)(w\eta_r^2)_r=\alpha r^{-\alpha-2}\left[2+\alpha+(1-\alpha)r\right]\left[(1-\alpha)+\alpha r^{-1}\right]^2.
\end{equation*}
From \eqref{01.17.7}, and the inequalities
\begin{equation*}
\begin{split}
&2s\iint_{\La_T^3}\Theta w\psi_r^2(w\eta_r)r dr dt-s\iint_{\La_T^3} \Theta w\psi_r^2\eta_rw_r dr dt\geq Cs\iint_{\La_T^3}\Theta \psi_r^2 dr dt
\end{split}
\end{equation*}
and
\begin{equation*}
\begin{split}
s\iint_{\La_T^3}\Theta w\psi\psi_r(w\eta_r){rr} dr dt\geq -C\iint_{\La_T^3}\Theta \psi_r^2 dr dt-Cs^2\iint_{\La_T^3}\Theta \psi^2 dr dt
\end{split}
\end{equation*}
where the constant $C>0$ depends only on $\alpha$ and $b$, we obtain
\begin{equation*}
\begin{split}
E_3
&\geq Cs\iint_{\La_T^3}\Theta \psi_r^2 dr dt+Cs^3\iint_{\La_T^3}  \Theta^3\psi^2 dr dt
\end{split}
\end{equation*}
for $s\geq C\sqrt{\gamma}\max\{1,T^{16}\}$, where the constant $C>0$ depends only on $\alpha$ and $b$.

Fourthly , from above steps,  we get
\begin{equation*}
	\begin{split}
		(P_1\psi,P_2\psi)_{L^2(\La_T)}
		&\geq Cs\iint_{\La_T}\Theta w\psi_r^2\df r\df t+Cs^3\iint_{\La_T}\Theta^3r^{2-\al}\psi^2\df r\df t\\
		&\hspace{4.5mm}-Cs\iint_{\La_T^2}\Theta w\psi_r^2\df r\df t-Cs^3\iint_{\La_T^2}\Theta^3\psi^2\df r\df t,
	\end{split}
\end{equation*}
where the constant $C>0$ depends only on $\alpha, a, b$, and $\eta$. Hence, from $\psi_r=\partial_r(e^{-s\xi}\varphi_{i,n})=e^{-s\xi}(\partial_r\varphi_{i,n}+s\Theta \eta_r\varphi_{i,n})$, we obtain
 \begin{equation}\label{01.17.8}
	\begin{split}
		&s\iint_{\La_T}\Theta w(\pt_r\vp_{i,n})^2e^{-2s\xi}\df r\df t+s^3\iint_{\La_T}\Theta^3r^{2-\al}\vp_{i,n}^2e^{-2s\xi}\df r\df t\\
		&\leq \left\|e^{-s\xi}f_{i,n}\right\|_{L^2(\La_T)}^2+ Cs\iint_{\La_T^2}\Theta w(\pt_r\vp_{i,n})^2e^{-2s\xi}\df r\df t+Cs^3\iint_{\La_T^2}\Theta^3\vp_{i,n}^2e^{-2s\xi}\df r\df t
	\end{split}
\end{equation}
for $s\geq C\gamma^2\max\{1,T^{16}\}$,
where the constant $C>0$ depends only on $\alpha, a, b$, and $\eta|_{(a,b)}$.

Finally, choose  $\zeta\in C_0^\infty(\mathbb{R}), 0\leq \zeta\leq 1$ such that
\begin{equation*}
\zeta=1 \mbox{ on }\left(\frac{2a+b}{3}, \frac{a+2b}{3}\right),\quad \zeta=0 \mbox{ on }\mathbb{R}-(a,b), \quad |\zeta'|\leq C\frac{1}{b-a} \mbox{ on }\mathbb{R}.
\end{equation*}
Denote $h=\zeta \psi$. Then, it follow from (see \eqref{01.17.2})
 \begin{equation*}
	\begin{split}
		&\iint_{\La_T}\pt_t\vp_{i,n} \left(\zeta \Theta \vp_{i,n}e^{-2s\xi}\right)\df r\df t\\
		&\geq -CT\iint_{\La_T}\zeta \Theta^\f{5}{4} \vp_{i,n}^2e^{-2s\xi}\df r\df t-Cs\ga T\iint_{\La_T} \zeta\Theta^\f{9}{4}\vp_{i,n}e^{-2s\xi}\df r\df t,
	\end{split}
\end{equation*}
and
\begin{equation*}
	\begin{split}
		&-s\iint_{\La_T}[\pt_r(w\pt_r\vp_{i,n})]\left(\zeta \Theta \vp_{i,n}e^{-2s\xi}\right)\df r\df t\\
		&\geq \f{1}{2}s\iint_{\La_T}\zeta\Theta w(\pt_r\vp_{i,n})^2e^{-2s\xi}\df r\df t-Cs^3\iint_{\La_T}\zeta \Theta^3\vp_{i,n}^2e^{-2s\xi}\df r\df t\\
		&\hspace{4.5mm}-\iint_{(a,b)\ts (0,T)}\Theta w(\pt_r\vp_{i,n})^2e^{-2s\xi}\df r\df t-Cs^2\iint_{(a,b)\ts (0,T)}\Theta \vp_{i,n}^2e^{-2s\xi}\df r\df t,
	\end{split}
\end{equation*}
we obtain 
 \begin{equation*}
	\begin{split}
		&\f{1}{2}a^rs\iint_{\La_T^2}\Theta  (\pt_r\vp_{i,n})^2e^{-2s\xi}\df r\df t\\
		&\leq \f{1}{2}s\iint_{\La_T}\zeta\Theta w(\pt_r\vp_{i,n})^2e^{-2s\xi}\df r\df t+n^2s\iint_{\La_T}\zeta \Theta \vp_{i,n}^2e^{-2s\xi}\df r\df t\\
		&\leq s\iint_{\La_T}\left(\pt_t\vp_{i,n}-\pt_r(w\pt_r\vp_{i,n})+n^2\vp_{i,n}\right)(\zeta\Theta \vp_{i,n}e^{-2s\xi})\df r\df t+Cs^3\iint_{\La_T}\zeta \Theta^3\vp_{i,n}^2e^{-2s\xi}\df r\df t\\
		&\hspace{4.5mm}+ \iint_{(a,b)\ts (0,T)}\Theta w(\pt_r\vp_{i,n})^2e^{-2s\xi}\df r\df t+Cs^2\iint_{(a,b)\ts (0,T)}\Theta \vp_{i,n}^2e^{-2s\xi}\df r\df t\\
		&\hspace{4.5mm}+CsT\iint_{\La_T}\zeta\Theta^\f{5}{4}\vp_{i,n}^2e^{-2s\xi}\df r\df t+Cs^2\ga T\iint_{\La_T}\zeta \Theta^\f{9}{4}\vp_{i,n}^2e^{-2s\xi}\df r\df t
	\end{split}
\end{equation*}
where the constant $C>0$ depends only on $\alpha, a, b$, and $\eta|{(a,b)}$.  Thus, drawing from \eqref{01.17.8}, we obtain the following inequality
 \begin{equation}\label{01.17.9}
	\begin{split}
		&s\iint_{\La_T}\Theta w(\pt_r\vp_{i,n})^2e^{-2s\xi}\df r\df t+s^3\iint_{\La_T}\Theta^3r^{2-\al}\vp_{i,n}^2e^{-2s\xi}\df r\df t\\
		&\leq C\left\|e^{-s\xi}f_{i,n}\right\|_{L^2(\La_T)}^2+Cs^3\iint_{(a,b)\ts (0,T)}\Theta^3\vp_{i,n}^2e^{-2s\xi}\df r\df t
	\end{split}
\end{equation}
for $s\geq C\gamma^2\max\{1,T^{16}\}$, where the constant $C>0$ depends only on $\alpha, a, b$, and $\eta|_{(a,b)}$.  
Consequently, we arrive at the following theorem.

 \begin{theorem}\label{01.17.T1}
Let $0 < a < b \leq 1$ and $T > 0$. Suppose $\eta$ is defined as in Notation \ref{01.17.N1}. Then, there exist positive constants $C > 0$ and $s_0 \geq 1$, which depend solely on $\alpha$, $a$, $b$, and $\eta|{(a,b)}$, such that for all $n \in \mathbb{N}$ and all $s \geq s_0$, the following inequality holds:
\begin{equation}\label{01.17.10}
		\begin{split}
			&s\iint_{\La_T}\Theta w(\pt_r\vp_{i,n})^2e^{-2s\xi}\df r\df t+s^3\iint_{\La_T}\Theta^3r^{2-\al}\vp_{i,n}^2e^{-2s\xi}\df r\df t\\
			&\leq C\left\|e^{-s\xi}f_{i,n}\right\|_{L^2(\La_T)}^2+Cs^3\iint_{(a,b)\ts (0,T)}\Theta^3\vp_{i,n}^2e^{-2s\xi}\df r\df t,
		\end{split}
	\end{equation}
where $\varphi_{i,n}$ is the solution of \eqref{01.17.1}.
\end{theorem}
\begin{proof}
By applying \eqref{01.17.9} with $\gamma = |\eta|_{\infty} + 1$, the proof of this theorem is concluded.
\end{proof}
 \begin{corollary}\label{01.17.C1}
Under the conditions specified in Theorem \ref{01.17.T1}, when $f_{i,n}=0$ for all $n\in\mathbb{N}$ and $i = 1,2$, there exists a positive constant $C$, which depends solely on $\alpha$, $a$, $b$, and $\eta|{(a,b)}$, such that for all $s\geq s_0$, the following inequality holds:
\begin{equation*}
\int_{\Lambda}|\varphi_{i,n}(r,T)|^2\mathrm{d}r\leq CT^7e^{Cs_0\frac{1}{T^8}}\iint_{(a,b)\times(0,T)}\varphi_{i,n}^2\mathrm{d}r\mathrm{d}t,
\end{equation*}
where the positive constant $C$ depends only on $\alpha$, $a$, $b$, and $\eta|_{(a,b)}$, and $s_0\geq 1$ is defined in Theorem \ref{01.17.T1}.
\end{corollary}

\begin{proof}
By multiplying both sides of \eqref{01.17.1} by $\varphi_{i,n}$ and integrating over $\Lambda$, we obtain
\begin{equation*}
\begin{split}
\frac{1}{2}\frac{\mathrm{d}}{\mathrm{d}t}\int_{\Lambda}\varphi_{i,n}^2\mathrm{d}r+\int_{\Lambda}w\left(\partial_r\varphi_{i,n}\right)^2\mathrm{d}r + n^2\int_{\Lambda}\varphi_{i,n}^2\mathrm{d}r = 0.
\end{split}
\end{equation*}
This implies that
\begin{equation*}
\int_{\Lambda}\varphi_{i,n}^2(t)\mathrm{d}r\leq\int_{\Lambda}\varphi_{i,n}^2(0)\mathrm{d}r\quad\text{for all }t\in(0,T].
\end{equation*}
Consequently, by Remark \ref{01.06.R1},   \eqref{01.17.10} and the facts:  
\begin{equation*}
\Theta e^{-2s\xi}\geq\frac{C}{T^8}e^{-Cs_0\frac{1}{T^8}}\quad\text{for }t\in\left(\frac{1}{4}T,\frac{3}{4}T\right), \;  
x^3e^{-2s_0x}\leq C\quad\text{for }x\in(0,\infty).
\end{equation*}
we get 
\begin{equation*}
		\begin{split}
			\int_\La \vp_{i,n}^2(T)\df r
			&\leq \f{2}{T}\int_{\f{1}{4}T}^{\f{3}{4}T}\int_\La \vp_{i,n}^2\df r\df t\leq \f{C}{T}\int_{\f{1}{4}T}^{\f{3}{4}T}\int_\La w(\pt_r\vp_{i,n})^2\df r\df t\\
			&\leq \f{C}{T}T^8 e^{Cs_0\f{1}{T^8}}s_0\int_{\f{1}{4}T}^{\f{3}{4}T}\int_\La \Theta w(\pt_r\vp_{i,n})^2e^{-2s\xi}\df r\df t\\
			&\leq CT^7e^{Cs_0\f{1}{T^8}}s_0^3\iint_{(a,b)\ts (0,T)}\Theta^3\vp_{i,n}^2e^{-2s_0\xi}\df r\df t\\
			&\leq CT^7e^{Cs_0\f{1}{T^8}}s_0^3\iint_{(a,b)\ts (0,T)} \vp_{i,n}^2\df r\df t,
		\end{split}
	\end{equation*}
where the positive constant $C>0$ depends only on $\alpha$, $a$, $b$, and $\eta|_{(a,b)}$.
\end{proof}

 \begin{corollary}\label{01.18.C1}
Let $\varphi$ be the weak solution of \eqref{12.14.1}. Then, the following observability inequality holds:
\begin{equation}\label{01.18.1}
\int_{\Omega} \varphi(z,T)^2 dz \leq CT^7e^{\frac{C}{T^8}} \int_0^T \int_{\mathbb{T} \times (a,b)} \varphi^2(z,t) dz dt,
\end{equation}
where the constant $C > 0$ depends solely on $\alpha$, $a$, and $b$.
\end{corollary}
\begin{proof}
Denote $\varphi = \varphi(z,t) = \varphi(\theta,r,t)$.
From \eqref{01.16.3}, we obtain
 \begin{equation*}
		\vp(\theta, r,t)=\sum_{n\in\N} \vp_{1,n}(r,t)g_{1,n}(\theta)+\sum_{n\in\N^*}\vp_{2,n}(r,t)g_{2,n}(\theta),
	\end{equation*}
which leads to
\begin{equation*}
\int_{\Omega} \varphi^2(z,t) dz = \sum_{n \in \mathbb{N}, i=1,2} \int_{\Lambda} \varphi_{i,n}^2(r,t) dr
\end{equation*}
for all $t \in [0,T]$. Furthermore, from Corollary \ref{01.17.C1}, we have
\begin{equation*}
		\int_\Om \vp^2(z,t)\df z=\sum_{n\in\N,i=1,2}\int_\La \vp_{i,n}^2(r,t)\df r
	\end{equation*}
	for all $t\in [0,T]$, and from Corollary \ref{01.17.C1} we have
	\begin{equation*}
		\begin{split}
			\int_\Om \vp^2(z,T)\df z
			&=\sum_{n\in\N,i=1,2}\int_\La \vp_{i,n}^2(r,T)\df r\leq CT^7e^{\f{C}{T^8}}\sum_{n\in\N,i=1,2}\iint_{(a,b)\ts (0,T)}\vp_{i,n}^2\df r\df t\\
			&=CT^7e^{\f{C}{T^8}}\int_0^T\int_{\mbT\ts (a,b)}\vp^2(\theta,r,t)\df \theta \df r\df t,
		\end{split}
	\end{equation*}
where the constant $C > 0$ depends solely on $\alpha$, $a$, and $b$ (note that $\eta|_{(a,b)}$ indeed depends only on $(a,b)$).
\end{proof}

 The following lemma is adapted from \cite[Theorem 1.5, p. 1091]{Tenenbaum}.
 
\begin{lemma}\label{01.16.L1}
Suppose ${\Phi_n}{n\in\mathbb{N}}$ constitutes an orthonormal basis for $L^2(\mathbb{T})$, and let $m\in\mathbb{N}$. Consider a real sequence ${a_n}{n\in\mathbb{N}}$.
For any open interval $I\subset\mathbb{T}$, the following inequality holds:
\begin{equation*}
\sum_{n\leq\mu} |a_n|^2\leq Ce^{C\sqrt{\mu}}\int_I \left(\sum_{n\leq \mu}a_n\Phi_n\right)^2\mathrm{d}\theta,
\end{equation*}
where the positive constant $C$ depends solely on the length of $I$, denoted as $|I|$.
\end{lemma}
For each $n\in\mathbb{N}$ and $i = 1,2$, define
\begin{equation*}
H_{i,n}=L^2(\mathbb{T})\otimes \varphi_{i,n},
\end{equation*}
and
\begin{equation*}
E_j=\bigoplus_{n\leq j, i = 1,2} H_{i,n}.
\end{equation*}
\begin{theorem}\label{01.17.T2}
Let $\alpha \in (0,1)$ and $T>0$. Suppose $0<a<b\leq 1$ and $0<c<d\leq 1$. Let $\varphi$ be the solution of \eqref{12.14.1} with $f = 0$. Then, there exists a positive constant $C$, which depends only on $\alpha$, $a$, $b$, and $\eta|{(a,b)}$, such that for each $j\in\mathbb{N}$ and every $\varphi^0\in E_j$, the following estimate holds:
\begin{equation*}
\begin{split}
\int_{\Omega} \varphi^2(T)\mathrm{d}z\leq CT^7e^{Cs_0\left(2^j+\frac{1}{T^8}\right)}\int_0^T\int_a^b \int_c^d\varphi^2\mathrm{d}z\mathrm{d}t.
\end{split}
\end{equation*}
\end{theorem}

\begin{proof}
Let $\varphi_{i,n}^0 \in L^2(\Lambda)$, where $0 \leq n \leq 2^{2j}$ and $i = 1,2$, be functions such that
\begin{equation*}
\varphi^0 = \sum_{0 \leq n \leq 2^{2j}, i = 1,2} \varphi_{i,n}^0(r) g_{i,n}(\theta),
\end{equation*}
with $g_{i,n}$ defined by \eqref{01.16.3}.
Then, the solution to \eqref{12.14.1} with $f = 0$ is expressed as
\begin{equation*}
\varphi(\theta, r, t) = \sum_{0 \leq n \leq 2^{2j}, i = 1,2} \varphi_{i,n}(r, t) g_{i,n}(\theta),
\end{equation*}
where $\varphi_{i,n}$ is the solution to \eqref{01.17.1} with $f_{i,n} = 0$. This leads to the following inequalities:
	\begin{equation*}
		\begin{split}
			\int_\Om \vp^2(T)\df z
			&=\sum_{0\leq n\leq 2^{2j},i=1,2}\int_\La \vp_{i,n}^2(T)\df r\\
			&\leq CT^7e^{Cs_0\f{1}{T^8}} \int_0^T\int_a^b\sum_{0\leq n\leq 2^j, i=1,2}\vp_{i,n}^2(r,t)\df r\df t\\
			&\leq CT^7e^{Cs_0\left(2^{j}+\f{1}{T^8}\right)}\int_0^T\int_a^b\int_c^d\left(\sum_{0\leq n\leq 2^{2j}, i=1,2}\vp_{i,n}(r,t)g_{i,n}(\theta)\right)^2\df \theta\df r\df t\\
			&=CT^7e^{Cs_0\left(2^{j}+\f{1}{T^8}\right)}\int_0^T\int_a^b\int_c^d\vp^2\df \theta \df r\df t,
		\end{split}
	\end{equation*}
where the second inequality employed Corollary \ref{01.17.C1}, and the third inequality utilized Lemma \ref{01.16.L1}. This concludes the proof of the theorem.
\end{proof}

\section{Observability and controllability on measurable subset}\label{S5}

 In this section, we employ an analytical approach to establish an observability inequality (Theorem \ref{01.18.T1}) on measurable sets. As a direct consequence, we derive the null controllability of equation \eqref{12.14.1} under $L^\infty(Q)$ controls.
\begin{proposition}\label{01.18.P1}
Let $0 < a < b \leq 1$. Then, there exist two constants $C \geq 1$ and $\rho \in (0,1]$, depending solely on $\alpha$, $a$, and $b$, such that for any solution $\varphi$ of equation \eqref{12.14.1} with $f = 0$, the following estimate holds:
\begin{equation}\label{01.18.2}
\begin{split}
\left|\partial_\theta^j \partial_r^k \partial_t^l \varphi(\theta, r,t)\right| \leq \frac{Ce^{\frac{C}{t-s}}(j+k)!l!}{\rho^{j+k}\left(\frac{t-s}{2}\right)^l} |\varphi(\cdot,s)|_{L^2(\Omega)}, \quad \text{for all } j,k,l \in \mathbb{N},
\end{split}
\end{equation}
where $(\theta,r) \in \mathbb{T} \times (a,b)$ and $0 \leq s < t$.
\end{proposition}

 \begin{proof}
Denote $\omega = \mathbb{T} \times (a,b)$.
It suffices to prove the case $s = 0$. Clearly, if 
$\vp^0=\sum\limits_{n\in\N^*}\vp_n^0\Phi_n$ with $\vp_n^0=(\vp^0,\Phi_n)_{L^2(\Om)}$, then
	\begin{equation*}
		\begin{split}
			\vp(z,t)=\sum_{n\in\N^*}\vp_n^0 e^{-\la_nt} \Phi_n,
		\end{split}
	\end{equation*}
	 where $\lambda_n$ and $\Phi_n$ are defined in Notation \ref{01.05.N1}. We introduce a new function
	\begin{equation}\label{01.08.3}
		\begin{split}
			\vp(\theta, r, \tau,t)=\sum_{n\in\N^*}\vp_n^0 e^{-\la_nt+\sqrt{\la_n}\tau}\Phi_n(\theta,r),\ (\theta, r)\in\Om, \tau\in \R, t>0.
		\end{split}
	\end{equation}
	Then we have
	\begin{equation}\label{01.18.4}
		\vp(\theta, r, 0,t)=\vp(\theta,r,t),\mbox{ and } \pt_t^l\vp(\theta, r, \tau,t)=\sum_{n\in\N^*} \vp_n^0 (-\la_n)^l e^{-\la_nt+\sqrt{\la_n}\tau}\Phi_n(\theta, r).
	\end{equation}
	 Moreover, for each $t > 0$, the function $\partial_t^l \varphi(\theta, r, \tau,t)$ satisfies
	\begin{equation}\label{01.18.5}
		\left(-\mcA+\pt_{\tau\tau}\right)\pt_t^l \vp(\theta, r, \tau, t)=0 \mbox{ in } \Om\ts \R.
	\end{equation}
	
Let $R \in \left(0, \frac{1}{4} \min\{a, 1-b\}\right)$, and   define $B_R(\theta_0,r_0) = \{z \in \Omega \colon |\theta-\theta_0|^2 + |r-r_0|^2 < R^2\}$ and $D_R(z_0,\tau_0) = B_R(z_0) \times (\tau_0-R, \tau_0+R)$ with $z_0 = (\theta_0,r_0)$.
Since there exist finitely many $z_0 \in \mathbb{T} \times [a,b]$ such that the sets $B_R(z_0)$ cover $\mathbb{T} \times (a,b)$, we can assume that $\omega \subset B_R(z_0)$.
Note that $r^\alpha$ is analytic on $\mathbb{T} \times (0, 1)$. From the real analytic estimates of solutions to linear elliptic equations with real analytic coefficients \cite[Chapter 3]{John}, there exist constants $C \geq 1$ and $\rho \in (0,1)$, depending solely on $\alpha$, $a$, and $b$, such that any solution of equation \eqref{01.18.5} satisfies
	\begin{equation}\label{01.18.9}
		\begin{split}
			\left\|\pt_\theta^j\pt_r^k\pt_t^l\vp(\cdot, \cdot, \cdot, t)\right\|_{L^\iy(D_R(z_0,0))}\leq C\f{(j+k)!}{\rho^{j+k}}\left(\int_{D_{2R}(z_0,0)}\left|\pt_t^l \vp(\theta, r, \tau, t)\right|^2\df \theta \df r \df \tau\right)^\f{1}{2}.
		\end{split}
	\end{equation}
	Note that for each $t>0$ we have
	\begin{equation*}
		\begin{split}
			\int_{D_{2R}(z_0,0)}\left|\pt_t^l \vp(\theta, r, \tau, t)\right|^2\df \theta \df r\df \tau
			&\leq \int_{-2R}^{2R} \int_{B_{2R}(z_0)}\left|\pt_t^l\vp(\theta, r, \tau, t)\right|^2\df \theta\df r\df \tau\\
			&\leq \int_{-2R}^{2R}\int_\Om \left|\pt_t^l\vp(\theta, r, \tau, t)\right|^2\df \theta\df r\df\tau,
		\end{split}
	\end{equation*}
	and
	\begin{equation*}
		\begin{split}
			\int_\Om \left|\pt_t^l\vp(\theta, r, \tau, t)\right|^2\df \theta\df r\leq \sum_{n\in\N^*}(\vp_n^0)^2\la_n^{2l}e^{-2\la_n t+2\sqrt{\la_n}\tau},
		\end{split}
	\end{equation*}
	hence, for each $t>0$, we have
	\begin{equation}\label{01.18.6}
		\begin{split}
			&\int_{D_{2R}(z_0,0)}\left|\pt_t^l \vp(\theta, r, \tau, t)\right|^2\df \theta\df r\df \tau\leq \max_{n\in\N^*}\left\{\la_n^{2l}e^{-\la_nt}\right\}\max_{n\in\N^*}\left\{e^{-\la_nt+4R\sqrt{\la_n}}\right\}\sum_{n\in\N^*}(\vp_n^0)^2.
		\end{split}
	\end{equation}
	 For each $k\in\N^*$, we have $\max\limits_{\la>0} \la^{2k}e^{-\la t}=(\f{2}{t})^{2k}(\f{k}{e})^{2k}$. From the Stirling formula, we obtain
	\begin{equation*}
		m!\sim \left(\f{m}{e}\right)^m \sqrt{2\pi m}, \mbox{ as } m\ra +\iy,
	\end{equation*}
	then
	\begin{equation}\label{01.18.7}
		\begin{split}
			\max_{n\in\N^*}\left\{\la_n^{2l}e^{-\la_nt}\right\}\leq C\left(\f{2}{t}\right)^{2l}\left(l!\right)^2.
		\end{split}
	\end{equation}
	 Since  $\max\limits_{n\in\N^*}\{e^{-\la_nt+4R\sqrt{\la_n}}\}\leq e^\f{4R^2}{t}$, from \eqref{01.18.6} and \eqref{01.18.7}, we get
	\begin{equation}\label{01.18.8}
		\begin{split}
			\int_{D_{2R}(z_0,0)}\left|\pt_t^l \vp(\theta, r, \tau, t)\right|^2\df \theta\df r\df \tau\leq Ce^\f{4R^2}{t}\left(\f{2}{t}\right)^{2l}\left(l!\right)^2.
		\end{split}
	\end{equation}
	 This, combined with \eqref{01.18.9}, yields 
	\begin{equation}\label{01.18.10}
		\begin{split}
			\left\|\pt_\theta^j\pt_r^k\pt_t^l\vp(\cdot, \cdot, \cdot, t)\right\|_{L^\iy(D_R(z_0,0))}\leq \f{Ce^\f{2R^2}{t}(j+k)!l!}{\rho^{j+k}(\f{t}{2})^l}\|\vp^0\|_{L^2(\Om)}.
		\end{split}
	\end{equation}
In particular, \eqref{01.18.2} holds according to \eqref{01.18.4}. This completes the proof of the proposition. 
\end{proof}

 The subsequent lemmas are derived from either \cite[Lemmas 3.4 and 3.5]{Liu} or \cite{Apraiz}.
\begin{lemma}\label{01.18.L1}
Let $g: [a,a+s]\to\mathbb{R}$, where $a\in\mathbb{R}$ and $s>0$, be a real-analytic function that satisfies the condition
\begin{equation*}
\left|\frac{\mathrm{d}^k}{\mathrm{d} x^k}g(x)\right|\leq Mk! (s\rho)^{-k}, \quad \text{for all } x\in [a,a+s] \text{ and for all } k\in\mathbb{N},
\end{equation*}
for some constants $M>0$ and $\rho\in (0,1]$. Suppose $F\subseteq [a,a+s]$ is a measurable subset with positive measure. Then, there exist two constants $C\geq 1$ and $h\in (0,1)$, which depend solely on $\rho$ and $\frac{|F|}{s}$, such that
\begin{equation*}
|g|_{L^\infty(a,a+s)}\leq CM^{1-h}\left(\frac{1}{|F|}\int_F |g(x)|\mathrm{d} x\right)^h.
\end{equation*}
\end{lemma}
\begin{lemma}\label{01.18.L2}
Let $\Omega$ be a bounded domain in $\mathbb{R}^N$, where $\R^N, N\in\N^*$. Let $f: \Omega\to \mathbb{R}$ be a real-analytic function that satisfies the condition
\begin{equation*}
\left|\partial_x^\beta f(x)\right|\leq M |\beta|! \rho^{-|\beta|}, \quad \text{for all } x\in \Omega \text{ and for all } \beta\in \mathbb{N}^N
\end{equation*}
for some constants $M>0$ and $\rho\in (0,1)$. Assume $\widetilde{\Omega}\subseteq \Omega$ is a measurable subset with positive measure. Then, there exist two constants $C\geq 1$ and $h\in (0,1)$, which depend solely on $|\Omega|, \rho,$ and $|\widetilde{\Omega}|$, such that
\begin{equation*}
|f|_{L^\infty(\Omega)}\leq CM^{1-h}\left(\int_{\widetilde{\Omega}}|f(x)|\mathrm{d} x\right)^h.
\end{equation*}
\end{lemma}
Next, we establish the following interpolation inequality.
\begin{proposition}\label{01.18.P2}
Let $0<a<b\leq 1$, $0\leq t_1<t_2<1$, $\eta\in (0,1)$, and $\sigma>0$. Assume $E\subseteq (t_1,t_2)$ is a measurable subset with $|E\cap (t_1,t_2)|\geq \eta(t_2-t_1)$, and for each $t\in E$, the measurable subset $D_t\subseteq \mathbb{T}\times (a,b)$ satisfies $|D_t|\geq \sigma$. Then, there exist constants $C\geq 1$ and $h\in (0,1)$, which depend solely on $\alpha, a, b, \eta,$ and $\sigma$, such that the solution of equation \eqref{12.14.1} with $f=0$ satisfies
\begin{equation}\label{01.18.11}
|\varphi(t_2)|_{L^2(\Omega)}\leq \left(\int_{t_1}^{t_2}\chi_E(t)|\varphi(t)|_{L^1(D_t)}\mathrm{d} t\right)^h \left(e^{\frac{C}{(t_2-t_1)^{8}}}|\varphi(t_1)|_{L^2(\Omega)}\right)^{1-h}.
\end{equation}
\end{proposition}

\begin{proof}
	Define $\om=\mbT\ts (a,b)$, and
	\begin{equation*}
		\begin{split}
			\tau=t_1+\f{\eta}{10}(t_2-t_1),\quad F=E\cap (\tau,t_2).
		\end{split}
	\end{equation*}
Then, it follows that $|F| > \frac{\eta}{2}(t_2 - t_1)$. According to Proposition \ref{01.18.P1}, there exist constants $C \geq 1$ and $h \in (0, 1)$, which depend solely on $\alpha$, $a$, $b$, $\eta$, and $\sigma$, such that for all $t \in [\tau, t_2]$ and for all $z \in \omega$, $j, k, l \in \mathbb{N}^*$, the following inequality holds:	
	\begin{equation}\label{01.18.12}
		\begin{split}
			\left|\pt_\theta^j\pt_r^k\pt_t^l\vp(\cdot, \cdot,t)\right|\leq \f{Ce^{\f{C}{t_2-t_1}}(j+k)!l!}{\rho^{j+k}(\f{\eta(t_2-t_1)}{2})^l}\|\vp(\cdot,t_1)\|_{L^2(\Om)}.
		\end{split}
	\end{equation}
	Let $M = e^{\frac{C}{t_2 - t_1}} |\varphi(\cdot, \cdot, t_1)|_{L^2(\Omega)}$. Then, for each $z \in \omega$ and for all $l \in \mathbb{N}$, we have
	\begin{equation*}
		\left|\pt_t^l \vp(\theta, r, t)\right|\leq Ml!\left(\f{\eta(t_2-\tau)}{20}\right)^{-l}, \mbox{ for all } t\in [\tau, t_2].
	\end{equation*}
	 Hence, by Lemma \ref{01.18.L1}, there exist constants $C \geq 1$ and $h \in (0, 1)$, depending only on $\eta$, such that for all $z \in \omega$, we obtain
	\begin{equation}\label{01.18.13}
		\|\vp(\theta, r, \cdot)\|_{L^\iy(\tau,t_2)}\leq M^{1-h}\left(\f{1}{|F|}\int_F|\vp(\theta, r, t)|\df t\right)^h.
	\end{equation}
	
	Now, from Corollary \ref{01.18.C1} and given that $0 \leq t_1 < t_2 < 1$, there exists a constant $C \geq 1$, depending only on $\alpha$, $a$, and $b$, such that
	\begin{equation}\label{01.18.14}
		\|\vp(z, t_2)\|_{L^2(\Om)}\leq \|\vp(z,\tau)\|_{L^2(\Om)}\leq e^{\f{C}{(t_2-\tau)^{8}}}\|\vp\|_{L^2(\mbT\ts (a,b)\ts (\tau,t_2))}.
	\end{equation}
	 From \eqref{01.18.12}, we can derive
	\begin{equation*}
		\begin{split}
		\|\vp\|_{L^2(\mbT\ts (a,b)\ts (\tau,t_2))}
		&\leq \|\vp\|_{L^\iy(\mbT\ts (a,b)\ts (\tau,t_2))}^\f{1}{2}\|\vp\|_{L^1(\mbT\ts (a,b)\ts (\tau,t_2))}^\f{1}{2}\\
		&\leq e^\f{C}{t_2-t_1}\|\vp(\cdot,\cdot,t_1)\|_{L^2(\Om)}^\f{1}{2}\|\vp\|_{L^1(\mbT\ts (a,b)\ts (\tau,t_2)}^\f{1}{2}. 
		\end{split}
	\end{equation*}
	 Thus, from \eqref{01.18.14}, we have
	\begin{equation}\label{01.18.15}
		\begin{split}
			\|\vp(\cdot,\cdot,t_2)\|_{L^2(\Om)}\leq \|\vp\|_{L^1(\mbT\ts (a,b)\ts (\tau,t_2))}^\f{1}{2}\left(e^{\f{C}{(t_2-t_1)^{8}}}\|\vp(\cdot,\cdot,t_1)\|_{L^2(\Om)}\right)^\f{1}{2}. 
		\end{split}
	\end{equation}
	 Note that from \eqref{01.18.13}, we get
	\begin{equation*}
		\begin{split}
			\|\vp\|_{L^1(\mbT\ts (a,b)\ts (\tau,t_2))}
			&\leq (t_2-\tau)\int_{\mbT\ts (a,b)}\|\vp(z,\cdot)\|_{L^\iy(\tau,t_2)}\df z\\
			&\leq CM^{1-h}\int_{\mbT\ts (a,b)}\left(\int_F |\vp(z,t)|\df t\right)^h \df z.
		\end{split}
	\end{equation*}
	This implies that
	\begin{equation}\label{01.18.16}
		\|\vp\|_{L^1(\mbT\ts (a,b)\ts (\tau,t_2))}\leq CM^{1-h}\left(\int_{\mbT\ts (a,b)}\int_F |\vp(z,t)|\df t\df z\right)^h
	\end{equation}
	by applying the  H\"older inequality.
	
	Finally,  using \eqref{01.18.12}, we obtain 
	\begin{equation*}
		\begin{split}
			\left|\pt_\theta^j\pt_r^k\vp(\cdot,\cdot, t)\right|\leq M(j+k)! \rho^{-(j+k)} \mbox{ for all } z\in \om \mbox{ and for all } t\in F.
		\end{split}
	\end{equation*}
Then, given that $|D_t| \geq \sigma$ for $t \in F$ and by Lemma \ref{01.18.L2}, there exist constants $C_1 \geq 1$ and $h_1 \in (0, 1)$, depending only on $\alpha$, $a$, $b$, $\rho$, and $\sigma$, such that for any $t \in F$,	
	\begin{equation}\label{01.18.17}
		\|\vp(\cdot,\cdot,t)\|_{L^\iy(\mbT\ts(a,b))}\leq C_1M^{1-h_1}\left(\int_{D_t}|\vp(z,t)|\df z\right)^{h_1}.
	\end{equation}
Combining \eqref{01.18.16} and \eqref{01.18.17} and applying the H\"older inequality, we have	
	\begin{equation*}
		\begin{split}
			\|\vp\|_{L^1(\mbT\ts (a,b)\ts (\tau,t_2))}
			&\leq
			CM^{1-h}\left(\int_FC_1M^{1-h_1}\left(\int_{D_t}|\vp(z,t)|\df z\right)^{h_1}\df t\right)^h\\
			&\leq CC_1M^{1-hh_1}\left(\int_F\int_{D_t}|\vp(z,t)|\df z\df t\right)^{hh_1}.
		\end{split}
	\end{equation*}
Combining this result with the definition of $M$ and \eqref{01.18.15}, we arrive at \eqref{01.18.11}. This completes the proof of the proposition.
\end{proof}

The following lemma is derived from \cite[Lemma 3.7]{Liu}.
\begin{lemma}\label{01.19.L1}
Let $E \subset (0,T)$ be a measurable subset with positive measure. Suppose $\ell \in (0,T)$ is a Lebesgue density point of $E$. Then, for every $q \in (0,1)$, there exists a sequence ${\ell_m}{m \in \mathbb{N}^*} \subset [0,T]$ that is monotonically decreasing and converges to $\ell$, such that for any $m \in \mathbb{N}^*$,
\begin{equation}\label{01.19.1}
\ell{m+1} - \ell_{m+2} = q(\ell_m - \ell_{m+1}),
\end{equation}
and
\begin{equation}\label{01.19.2}
\left|E \cap (\ell_{m+1}, \ell_m)\right| \geq \frac{\ell_m - \ell_{m+1}}{3}.
\end{equation}
\end{lemma}
With the above preparations, we can now proceed to prove Theorem \ref{01.18.T1}.

 \begin{proof}[Proof of Theorem \ref{01.18.T1}]
Without loss of generality, assume $T \in (0,1]$, and $(\mathbb{T} \times {0}) \cap \overline{D} = \emptyset$, and $\omega = \mathbb{T} \times (\delta, 1) \subset \Omega$ ($\delta > 0$ is a given constant) such that $D \subset \omega \times (0,T)$. For almost every $t \in [0,T]$, we define
\begin{equation*}
 D_t=\{z\in\om\colon (z,t)\in D\}, \quad E=\left\{t\in (0,T)\colon |D_t|\geq \f{|D|}{2T}\right\}.
\end{equation*}
It is evident that $E$ is a measurable set. By Fubini's Theorem, and
\begin{equation*}
|D| = \int_0^T |D_t| , dt = \int_E |D_t| , dt + \int_{(0,T) \setminus E} |D_t| , dt \leq |E||\omega| + \frac{|D|}{2},
\end{equation*}
we deduce that $|E| \geq \frac{|D|}{2|\omega|}$, and
\begin{equation}\label{01.19.4}
\chi_E(t)\chi_{D_t}(z) \leq \chi_D(z,t), \quad \text{for almost every } (z,t) \in \Omega \times (0,T).
\end{equation}
Let $\ell \in (0,T)$ be a Lebesgue density point of $E$. Then, from Lemma \ref{01.19.L1}, for each $q \in (0,1)$, there exists a monotone sequence ${\ell_m}{m \in \mathbb{N}^*}$ such that \eqref{01.19.1} and \eqref{01.19.2} hold. Moreover, we have
\begin{equation}\label{01.19.3}
\lim\limits_{m \to +\infty} \ell_m = \ell.
\end{equation}
From Proposition \ref{01.18.P2}, there exist constants $C \geq 1$ and $h \in (0,1)$, depending only on $\alpha, a, b, D$, and $T$, such that for all $m \in \mathbb{N}^*$, we have
	\begin{equation*}
		\begin{split}
			\|\vp(\cdot, \ell_m)\|_{L^2(\Om)}\leq \left(e^{\f{C}{(\ell_m-\ell_{m+1})^8}}\int_{\ell_{m+1}}^{\ell_{m}}\chi_E(t)\|\vp(\cdot,z)\|_{L^1(D_t)}\df t\right)^h \|\vp(\cdot,\ell_{m+1})\|_{L^2(\Om)}^{1-h}.
		\end{split}
	\end{equation*}
	 This, combined with Young's inequality 
	\begin{equation*}
		ab\leq \e a^p+\e^{-\f{r}{p}}b^r, \mbox{ for all } a>0, b>0, \e>0 \mbox{ and } \f{1}{p}+\f{1}{r}=1, p>1,r>1,
	\end{equation*}
	 yields
	\begin{equation*}
		\|\vp(\cdot,\ell_m)\|_{L^2(\Om)}\leq \e \|\vp(\cdot, \ell_{m+1})\|_{L^2(\Om)}+\e^{-\f{1-h}{h}}e^\f{C}{(t_2-t_1)^8}\int_{\ell_{m+1}}^{\ell_m}\chi_E(t)\|\vp(\cdot,t)\|_{L^1(D_t)}\df t
	\end{equation*}
	 for all $m \in \mathbb{N}$. Then, for all $m \in \mathbb{N}^*$, we have
	\begin{equation*}
		\begin{split} &\e^\f{1-h}{h}e^{-\f{C}{(\ell_m-\ell_{m+1})^8}}\|\vp(\cdot,\ell_m)\|_{L^2(\Om)}-\e^\f{1}{h}e^{-\f{C}{(\ell_m-\ell_{m+1})^8}}\|\vp(\cdot,\ell_{m+1})\|_{L^2(\Om)}\\
			&\leq \int_{\ell_{m+1}}^{\ell_m}\chi_E(t)\|\vp(\cdot,t)\|_{L^1(D_t)}\df t.
		\end{split}
	\end{equation*}
	Letting $\e=e^{-h\f{1}{(\ell_m-\ell_{m+1})^8}}$  in the latter inequality, we obtain
	\begin{equation*}
		\begin{split}
			&e^{-(C+1-h)\f{1}{(\ell_m-\ell_{m+1})^8}}\|\vp(\cdot,\ell_m)\|_{L^2(\Om)}-e^{-(C+1)\f{1}{(\ell_m-\ell_{m+1})^8}}\|\vp(\cdot, \ell_{m+1})\|_{L^2(\Om)}\\\
			&\leq \int_{\ell_{m+1}}^{\ell_m}\chi_E(t)\|\vp(\cdot,t)\|_{L^1(D_t)}\df t
		\end{split}
	\end{equation*}
for all $m \in \mathbb{N}^*$. Choosing $q = \left(\frac{C+1-h}{C+1}\right)^{\frac{1}{8}}$, from \eqref{01.19.1}, we get	
	\begin{equation*}
		\begin{split} &e^{-(C+1-h)\f{1}{(\ell_m-\ell_{m+1})^8}}\|\vp(\cdot,\ell_m)\|_{L^2(\Om)}-e^{-(C+1-h)\f{1}{(\ell_{m+1}-\ell_{m+2})^8}}\|\vp(\cdot,\ell_{m+1})\|_{L^2(\Om)}\\
			&\leq  \int_{\ell_{m+1}}^{\ell_m}\chi_E(t)\|\vp(\cdot,t)\|_{L^1(D_t)}\df t
		\end{split}
	\end{equation*}
	  for all $m \in \mathbb{N}^*$. Summing over $m \in \mathbb{N}^*$, from \eqref{01.19.3} and $\sup{t \in [0,T]} |\varphi(\cdot, t)|_{L^2(\Omega)} < +\infty$, we obtain
	\begin{equation*}
		\begin{split}
			\|\vp(\cdot, \ell_1)\|_{L^2(\Om)}\leq e^{(C+1-h)\f{1}{(\ell_1-\ell_2)^8}}\int_\ell^{\ell_1}\chi_E(t)\|\vp(\cdot,t)\|_{L^1(D_t)}\df t.
		\end{split}
	\end{equation*}
	 Combining this with \eqref{01.19.4} and $|\varphi(\cdot, T)|_{L^2(\Omega)} \leq |\varphi(\cdot, \ell_1)|_{L^2(\Omega)}$, we arrive at \eqref{01.18.18}. This completes the proof of the theorem.
\end{proof}

\subsection{Controllability on measurable set}

 Below, we demonstrate that system \eqref{12.14.1} exhibits null controllability with an $L^\infty$ control applied on a measurable subset $D \subseteq Q$. Despite the standard nature of this result (HUM), we include the proof for thoroughness.
 
\begin{proof}[Proof of Theorem \ref{01.19.T1}]
Define $Y = \{\chi_D y\}$, where $y$ represents the solution to the following equation:
	\begin{equation}\label{01.19.7}
		\begin{cases}
			\pt_ty-\mcA y=0, &\mbox{in }Q, \\
			y=0, &\mbox{on }\pt Q, \\
			y(T)=y^T, &\mbox{in }\Om
		\end{cases}
	\end{equation}
	with $y^T \in L^2(\Omega)$. Clearly, $Y$ is a subspace of $L^1(D)$. We define the functional 
	\begin{equation*}
		F: Y\ra \R, \quad F(\chi_Dy)=-\int_\Om \vp^0y(0)\df z.
	\end{equation*}
	 From Theorem \ref{01.18.T1}, we derive 
	\begin{equation*}
		\begin{split}
			|F(\chi_Dy)|\leq \|\vp^0\|_{L^2(\Om)}\|y(0)\|_{L^2(\Om)}\leq C\|\vp^0\|_{L^2(\Om)}\|y\|_{L^1(D)},
		\end{split}
	\end{equation*}
indicating that $F$ is a bounded linear functional on $Y$. By the Hahn-Banach theorem, there exists a linear extension $\widetilde{F}: L^1(D) \to \mathbb{R}$ such that
\begin{equation*}
\widetilde{F} = F \text{ on } Y, \quad \text{and} \quad |\widetilde{F}(g)| \leq C |\varphi^0|_{L^2(\Omega)} |g|_{L^1(D)} \text{ for all } g \in L^1(D).
\end{equation*}
This implies that  $\wt F\in L^\iy(D)$, the dual space of $L^1(D)$. Consequently, there exists $f \in L^\infty(D)$ such that
\begin{equation*}
|f|_{L^\infty(D)} \leq C |\varphi^0|_{L^2(\Omega)}, \quad \text{and} \quad \widetilde{F}(g) = -\iint_D fg  dz dt \text{ for all } g \in L^1(D).
\end{equation*}
Extend $f$ to be zero on $Q \setminus D$, and denote it by the same symbol.

Finally, since $\varphi$ is the weak solution of \eqref{12.14.1} and $y$ is the weak solution of \eqref{01.19.7} (i.e., $y \in L^2(0,T; H_0^1(\Omega; w))$), we have
	\begin{equation*}
		\begin{split}
			\lg\pt_t\vp, y(t)\rg_{H^{-1}(\Om;w), H_0^1(\Om;w)}+\int_\Om A\nabla\vp\cdot\nabla y(t)\df z=\int_\Om \chi_{D_t}fy(t)\df z \mbox{ for a.e.}\ t\in [0,T].
		\end{split}
	\end{equation*}
	Integrating over $[0,T]$ and applying integration by parts, noting that $\varphi \in L^2(0,T; H_0^1(\Omega; w))$, we obtain
	\begin{equation*}
		\begin{split}
			\int_\Om \vp(T)y(T)\df z-\int_\Om \vp(0)y(0)\df z=\iint_Q \chi_Dfy\df z\df t.
		\end{split}
	\end{equation*}
	 Since $\chi_D y \in L^1(D)$, it follows that $\int_\Omega \varphi(T) y^T  dz = 0$ for all $y^T \in L^2(\Omega)$. This implies that $\varphi(T) = 0$ for a.e. $z \in \Omega$. Thus, we complete the proof of this theorem.
\end{proof}
 
Once Theorem \ref{01.19.T1} is established, we can proceed to discuss the bang-bang property for time and norm optimal control problems. However, as this is a standard procedure, we omit the detailed discussion in this paper.
 
\section*{Declarations}
The authors have not disclosed any competing interests
and data availability

\end{document}